\documentclass[reqno]{amsart}
\usepackage{amsmath, amsthm, amscd, amsfonts, amssymb, graphicx, xcolor,mathrsfs}
\usepackage{multicol,caption}
\usepackage{hyperref}
\usepackage{wrapfig}
\usepackage{lipsum}
\usepackage{comment}
\usepackage{transparent}
\usepackage[english]{babel}
\hypersetup{colorlinks=true,allcolors=blue}

\newtheorem{theorem}{Theorem}[section]
\newtheorem{corollary}[theorem]{Corollary}
\newtheorem{lemma}[theorem]{Lemma}
\newtheorem{proposition}[theorem]{Proposition}
\newtheorem{question}[theorem]{Question}

\theoremstyle{definition}
\newtheorem{notation}[theorem]{Notation}
\newtheorem*{theorem*}{Theorem}
\theoremstyle{definition}
\newtheorem{definition}[theorem]{Definition} 
\newtheorem{example}[theorem]{Example}

\theoremstyle{plain} 
\newtheorem*{claim}{Claim} 

\usepackage{xpatch}
\xpatchcmd{\proof}{\itshape}{\scshape}{}

\title[]{Hanf Numbers for Poset Games}
\author[]{Fabián Rivero Herrera}
\email{Fabian.Rivero.Herrera@vutbr.cz}
\address{Institute of Mathematics, Faculty of Mechanical Engineering, Brno University of
Technology, Brno, Czech Republic}
\begin{document}
\begin{abstract}
Given two well partial orders $(P;\leq_P)$ and $(T;\leq_T)$, each with a minimum element, we study the following question: which player has a winning strategy for Chomp on the poset $(P\times T;\leq_{P\times T})$? Here, $(P\times T;\leq_{P\times T})$ denotes the poset obtained as the Cartesian product of $P$ and $T$, equipped with the corresponding lexicographic order. The answer to this very natural question depends strongly on the specific choice of $(P;\leq_P)$ and $(T;\leq_T)$. For this reason, we restrict our attention to classes of posets given by powers of a fixed poset: $\{(P^\sigma;\leq_{P^\sigma})\mid \sigma\in\mathrm{Ord}\}$. A fundamental fact about these classes of structures is that, if the second player does not have a winning strategy for all the posets in $\{(P^\sigma;\leq_{P^\sigma})\mid \sigma\in\mathrm{Ord}\}$, there exists an ordinal $\xi$ such that the second player has a winning strategy on $(P^{\xi};\leq_{P^{\xi}})$ but not on $(P^{\gamma};\leq_{P^{\gamma}})$ for all $\gamma\geq\xi+1$. Determining the corresponding ordinal for this Hanf number-style property constitutes the main objective of this work.

Inspired by results of Garc\'ia-Marco and Knauer, we focus on classes of posets with a purely algebraic definition. These posets arise from submonoids (with respect to the natural sum, or Hessenberg sum) of ordinals of the form $\omega^\sigma$ and are generated by sets of ordinals. In the process, we provide a test to determine whether a finite set $\Gamma$ of ordinals indeed yields well partial orders, and, using set-theoretic techniques, we establish an upper bound for the ordinal $\xi$: if $\Gamma\subset\omega$, then $\xi<\omega_1$, and otherwise $\xi<|\bigcup\Gamma|^+$.
\end{abstract}
\maketitle
\vspace{-1cm}
\section{Introduction}
In a poset game, two players take turns choosing elements from a partially ordered set and, on each turn, removing the chosen element and all those above it. The first player to run out of elements to choose loses. A particular case (and one of the most popular poset games) is the game \textit{Chomp}: the poset is a rectangle of dimension $n\times m$ where each element is smaller than all the elements above it and to its right; the first player to eliminate the minimum (the element at the bottom left) loses. We can also define this game on any poset with a global minimum and, of course, by removing it, we obtain a poset game with the rules we explained at the beginning. The history of this combinatorial game dates back to 1952, when the Dutch mathematician Frederik Schuh introduced the ``game of divisors’’. Quite some time later, David Gale described in \cite{gale} the Chomp game in the way just mentioned.

A natural question one may ask is the following: given two posets $(P;\leq_P) $ and $(T;\leq_T)$, each with a global minimum, what can we say about the Chomp game on the poset $(P\times T;\leq_{P\times T})$? Here, $P\times T$ denotes the set of pairs $(p,q)$ with $p\in P$ and $q\in T$, and $\leq_{P\times T}$ is the lexicographic order; that is, $(p,q)\leq_{P\times T}(p',q')$ iff $p\leq_P p'$ or $p=p'$ and $q\leq_T q'$. More specifically, a similar question can be asked for successive powers $(P^{n+1};\leq_{P^{n+1}}):=(P\times P^{n};\leq_{P\times P^{n}})$ of $(P;\leq_P)$, with $n\geq 1$.

We will see that the answer to these questions depends heavily on the posets $(P;\leq_P)$ and $(T;\leq_T)$ that we choose. We will focus on addressing the second question and, in particular, we will work with a generalization of the Chomp game that was introduced in \cite{ALCO_2018__1_3_371_0}. There, the Chomp game is studied on the posets given by numerical semigroups $ \mathcal{S}\subset\mathbb{N}$ with the order defined by $a\leq_{\mathcal{S}}b$ iff $b-a\in\mathcal{S}$. This is a huge generalization of the original game, so that, as in the classic Chomp, there is no known constructive way to find winning strategies, except in particular cases.

Specifically, we propose a generalization of the game studied in \cite{ALCO_2018__1_3_371_0} to a set-theoretic context: as in $\mathbb{N}$, any additively indecomposable ordinal has a monoid structure on which we can define a partial order and, thus, play the Chomp game. In particular, we will focus on submonoids with natural sum (also called Hessenberg sum). The following sections will make clear that this family of submonoids is very interesting for our purposes. For example, we prove that it is possible to apply results from \cite{ALCO_2018__1_3_371_0} to this context to obtain constructive winning strategies. Apart from that, we will see that the algebraic definition of these posets allows us to use purely set-theoretic techniques to draw some conclusions about the game.

Since this paper may attract readers interested in game theory or set theory, the following section presents some very basic facts from both disciplines that will be necessary later. In the third section, we present some basic results about Chomp on $(P\times T;\leq_{P\times T})$ and study some order-theoretic properties of the well partial order generators; that is, sets of ordinals that generate monoids in which we can define a natural well partial order structure. For instance, we provide a very general test to determine whether a finite set of ordinals is a well partial order generator. In the final section, we will investigate classes of well partial orders $\{ (\mathcal{S}^\tau;\leq_{\mathcal{S}^\tau}) \mid\tau\geq\sigma\}$ generated by sets of ordinals $\Gamma$, where the sets $\mathcal{S}^\tau\subseteq\omega^\tau$ are monoids under natural sum. In particular, if the second player does not have a winning strategy on all the posets of the corresponding class, then there exists a successor ordinal $\xi+1$, which we denote by $\text{ch}(\Gamma)$, such that the second player has a winning strategy on $(\mathcal{S}^\alpha;\leq_{\mathcal{S}^\alpha})$, for $\alpha\leq\xi$, but not on $(\mathcal{S}^{\gamma};\leq_{\mathcal{S}^{\gamma}})$ for all $\gamma\geq\xi+1$. Our main result states that, if $\Gamma$ is finite, $\text{ch}(\Gamma)<\omega_1$ if $\Gamma\subset\omega$ and $\textnormal{ch}(\Gamma)<|\bigcup\Gamma|^+$ otherwise.
\section{Preliminaries}
\subsection{Game theory} The game-theoretic concepts required in this work are fairly simple; however, for a reader not well acquainted with these topics, we shall briefly mention the tools that will be needed in what follows. From now on, we will refer to the first player as player $A$ and the second player as player $B$.
\begin{theorem}[Zermelo's theorem]
    In any finite and impartial two-player game there exists a winning strategy for either $A$ or $B$.
\end{theorem} Here, a finite game is a game in which any play of the game ends after a finite number of moves. In the case of games on partially ordered sets, we can ensure that the game is finite by requiring that there are no infinite antichains or infinite descending chains; that is, we shall work with \textit{well partial orders}.

\begin{proposition}
    If $ P $ is a well partial order with a global minimum and a global maximum (different from each other), player $A$ has a winning strategy for Chomp on $P$.
\end{proposition}
The proof of the preceding result uses what is known as \emph{strategy stealing}: the idea is that, if player \textit{B} had a winning strategy, then player \textit{A} could remove the maximum element in the first move, and thereafter any response by player \textit{B} is something player \textit{A} could have played as well. More details can be found in \cite{intro}. The following is also a fundamental argument in poset games:
\begin{proposition}\label{propsim}
    Let $ (T; \leq_T) $ be a well partial order. Then, player $B$ has a winning strategy for Chomp on the poset $  T^2 = (\{0,1\} \times T) \cup \{0\} $ with $0\leq_{T^2} (a,b)$ for all $(a,b)\in\{0,1\} \times T $ and $ (a,b) \leq_{T^2} (c,d) $ iff $  a = c $ and $ b \leq_T d $.
\end{proposition}\begin{proof} Player $  B $ simply has to copy player $A$'s moves but in the ``opposite branch" of the poset: if player $A$ chooses the element $(0,c) $, player $B$ will choose $(1,c)$ on the next move. In this way, the elements of the form $(a,b)$ will be exhausted just before one of player $A$'s turns, forcing player $A$ to choose $0 $ and lose. \end{proof}
The Sprague–Grundy theorem, which states that any impartial game of perfect information under the normal play convention (the last player able to move is the winner) is equivalent to playing a poset game where the poset is a well order (this is the popular game \textit{Nim} with one ``heap''), formalizes the idea of the previous proof. Although this is a very useful tool, we will not explicitly use it in the rest of the paper.
\subsection{Set theory} A reader unfamiliar with the basic properties of ordinals or other set-theoretic concepts may find what is not defined here, for example, in \cite{Jech2003}.

As is well known, the usual sum and product of ordinals lack certain useful properties that the sum and product of natural numbers possess. The natural sum and product of ordinals extend the sum and product of natural numbers and preserve commutativity and other properties of interest. Before the definition, it is worth recalling the following theorem:
\begin{theorem}[Cantor normal form for base $\omega$] For every non-zero ordinal $\alpha$, there is a unique positive natural number $k$, a unique strictly decreasing finite sequence of ordinals $\alpha_1>\cdots>\alpha_k$ and a unique sequence $n_1, \dots , n_k$ of positive natural numbers such that $$ \alpha=\omega^{\alpha_1}\cdot n_1+\cdots+\omega^{\alpha_k}\cdot n_k.$$
\end{theorem}
For obvious reasons, we will often refer to the \textit{terms} of an ordinal $\alpha$ to mean the corresponding ordinals $\omega^{\alpha_i}\cdot n_i$, and we will call the natural numbers $n_i$ the \textit{coefficients} of these terms.

\begin{definition}[Natural sum and product]
    Let $\alpha$ and $\beta$ be ordinals with the following Cantor normal form representation: $$\alpha=\omega^{\gamma_1}\cdot n_1+\cdots+\omega^{\gamma_k}\cdot n_k \ \text{ and } \ \beta=\omega^{\gamma_1}\cdot m_1+\cdots+\omega^{\gamma_k}\cdot m_k,$$ where some of the coefficients may be 0. Then, we define $$\alpha\oplus\beta=\omega^{\gamma_1}\cdot(n_1+m_1)+\cdots+\omega^{\gamma_k}\cdot(n_k+m_k).$$ If $\alpha=\omega^{\alpha_1}\cdot a_1+\cdots+\omega^{\alpha_k}\cdot a_k$ and $\beta=\omega^{\beta_1}\cdot b_1+\cdots+\omega^{\beta_l}\cdot b_l$, we define $$\alpha\otimes\beta=\bigoplus_{i\leq k, \ j\leq l}\omega^{\alpha_i\oplus\beta_j}\cdot a_i\cdot b_j.$$
\end{definition}
The proof of the following lemma, which states the main properties we will use later, can be found in \cite[Lemma 3.3]{DEJONGH1977195}. 
\begin{lemma}\label{lemamon}
    For every ordinal $\sigma$, the ordinal $\omega^\sigma$ is closed under $\oplus$ and $\omega^{\omega^\sigma}$ is also closed under $\otimes$. Moreover, for all ordinals $\alpha,\beta$ and $\gamma$,\begin{itemize}
    \item $(\alpha\oplus\beta)\oplus\gamma=\alpha\oplus(\beta\oplus\gamma)$ and $(\alpha\otimes\beta)\otimes\gamma=\alpha\otimes(\beta\otimes\gamma)$,
    \item $ \alpha\oplus\beta=\beta\oplus\alpha$ and $\alpha\otimes\beta=\beta\otimes\alpha$,
    \item $\alpha\otimes(\beta\oplus\gamma)=(\alpha\otimes\beta)\oplus(\alpha\otimes\gamma)$.
    \end{itemize}
\end{lemma}
Therefore, for any ordinal of the form $\omega^\sigma$, we have a monoid structure $(\omega^\sigma; \oplus, 0)$.

\begin{notation} Given a set $\Gamma\subset\omega^\sigma$ of ordinals, we denote $$ \mathcal{S}^\sigma:=\langle \Gamma\rangle^\sigma =\{\alpha_1\otimes \gamma_1 \ \oplus \ \cdots \ \oplus \alpha_n\otimes \gamma_n \mid \gamma_i\in\Gamma \text{ and } \alpha_i\otimes\gamma_i\in\omega^\sigma\}.$$
\end{notation}
By \hyperref[lemamon]{Lemma \ref{lemamon}}, $(\mathcal{S}^\sigma;\oplus,0)$ is a submonoid of $(\omega^\sigma;\oplus,0)$. For example, if $ \Gamma=\{p_1,\dots,p_n\}\subset\omega$, any $\beta\in\mathcal{S}^\sigma$ can be written as \begin{equation}\label{ecdes}\beta=\omega^{\alpha_1}\cdot(p_1\cdot b_1^1+\cdots+p_n\cdot b_n^1)+\cdots+\omega^{\alpha_m}\cdot(p_1\cdot b_1^m+\cdots+p_n\cdot b_n^m), 
\end{equation}for some ordinals $\sigma>\alpha_1>\cdots>\alpha_m$ and $b_i^j\in\omega $.
\begin{definition}
    Given an ordinal $\alpha = \omega^{\alpha_1} \cdot a_1 + \cdots + \omega^{\alpha_n} \cdot a_n$, we define the function $\overline{e}: \text{Ord} \to \text{Ord}$ by $\overline{e}(\alpha) = \alpha_1$ and $\overline{e}(0)=0$. For a set $\Gamma$ of ordinals, we define $\overline{e}(\Gamma) = \bigcup_{\gamma\in\Gamma}\overline{e}(\gamma) $.
\end{definition}
Thus, the least ordinal $\sigma$ such that $\Gamma \subset \omega^{\sigma}$ is $\overline{e}(\Gamma)+1$ if $\overline{e}(\Gamma) \in \Gamma$ (for instance, if $\Gamma$ is finite), and $\overline{e}(\Gamma)$ otherwise.

Regarding the poset structure, we will need a notion of subtraction to define the order on the submonoids. It is well known that we cannot define a (left and right) subtraction on ordinals; however, the following notion will be sufficient for our purposes here.

\begin{definition}[Ordinal subtraction]
    For all ordinals $\alpha$ and $\beta$ such that $\beta\geq\alpha$, we define $\beta-\alpha$ as the unique ordinal $\gamma$ satisfying $\beta=\alpha+\gamma$.
\end{definition}
Equivalently, we can define $\beta-\alpha:=\text{ot}(\beta\backslash\alpha)$ (where ``$\text{ot}$'' denotes the order type), which is perhaps a more natural definition. It is not difficult to check that for the case $\sigma = 1 $, the following lemma provides the same poset structure as when we use the usual subtraction of natural numbers.
\begin{lemma}\label{lemposetstr}
    Let $\Gamma\subset\omega^\sigma$ be a set of ordinals. Then, $\mathcal{S}^\sigma=\langle \Gamma\rangle^\sigma$ has a natural poset structure given by $\alpha\leq_{\mathcal{S}^\sigma}\beta $ iff $\alpha\in\beta $ and $\beta-\alpha\in\mathcal{S}^\sigma $.
\end{lemma}\begin{proof} Reflexivity and antisymmetry are straightforward. To prove transitivity, take $\alpha,\beta,\delta\in\mathcal{S}^\sigma$ such that $\alpha\leq_{\mathcal{S}^\sigma}\beta$ and $\beta\leq_{\mathcal{S}^\sigma}\delta$. Then, there are $\zeta,\xi\in\mathcal{S}^\sigma$ satisfying $\beta=\alpha+\zeta$ and $\delta=\beta+\xi$, so $\delta=\alpha+(\zeta+\xi)$. If the coefficients of the ordinals in \(\Gamma\) are denoted by \(g_i\), the coefficients of \(\alpha\) will be sums of natural numbers of the form \(g_i \cdot a_i\). And similarly for $\beta$, $\delta$, $\zeta$ and $\xi$, with $b_i$, $d_i$, $z_i$ and $x_i$ instead of $a_i$, respectively. Therefore, the $d_i$'s are of the form $a_i$, $a_i+z_i$, $a_i+z_i+x_i $, $z_i$, $z_i+x_i $ or $x_i$. If we replace the \(a_i\)'s with 0, we obtain precisely the ordinal \(\zeta + \xi\). That is, if \(\delta = \delta_1 \otimes \gamma_1 \oplus \cdots \oplus \delta_n \otimes \gamma_n\), by replacing the \(a_i\)'s in the \(d_i\)'s of the ordinals \(\delta_j\) with 0, we obtain new ordinals \(\delta_j'\) such that 
\(
\delta_1' \otimes \gamma_1 \oplus \cdots \oplus \delta_n' \otimes \gamma_n = \zeta + \xi
\), which proves that $\alpha\leq_{\mathcal{S}^\sigma}\delta$.
 \end{proof}
To slightly simplify the notation, we will denote the posets by $(\mathcal{S}^\sigma;\leq_{\mathcal{S}^\sigma})$ instead of $(\langle \Gamma \rangle^\sigma;\leq_{\langle \Gamma \rangle^\sigma})$.

\section{Chomp on $(P\times T;\leq_{P\times T})$}
\subsection{Basic results} As mentioned in the introduction, we will work with posets that possess a global minimum and, specifically, with posets that correspond to the Cartesian product of two posets endowed with the lexicographic order. The first thing we can observe is the following simple result:
\begin{proposition}\label{propsen}
Let $(P;\leq_P)$ and $(T;\leq_T)$ be well partial orders with global minima. If player $A$ has a winning strategy for Chomp on $T$, then player \textit{A} also has a winning strategy for Chomp on $(P\times T;\leq_{P\times T})$.
\end{proposition}\begin{proof}If $0$ is the global minimum of $P$ and $t$ is a winning move for \textit{A} when playing on $(T;\leq_T)$, then $(0,t)$ is a winning move for \textit{A} on $(P\times T;\leq_{P\times T})$: the remaining poset is isomorphic to the poset that results after the move \(t\) in \((T; \leq_T)\).
\end{proof}
Apart from the above, in general we cannot say much more:
\begin{example}
    Consider the following posets:
\begin{figure}[h!]
\def\svgwidth{1\columnwidth}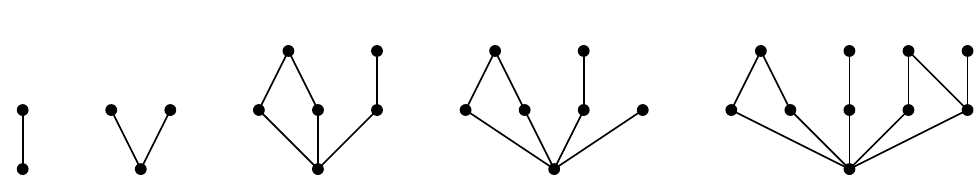  \caption{Hasse diagram of some posets.}\label{posets}  
\end{figure}
\begin{enumerate}
    \item If $P$ is $P_1$ and $T$ is $P_2$, then player \textit{A} has the winning strategy on $(P;\leq_P)$, player \textit{B} has it on $(T;\leq_T)$, and $(1,0)$ is a winning move for player \textit{A} on $(P\times T;\leq_{P\times T})$.
    \item If $P$ is $P_4$ and $T$ is $P_2$, then player \textit{A} has the winning strategy on $(P;\leq_P)$, player \textit{B} has it on $(T;\leq_T)$, and player \textit{B} also has it on $(P\times T;\leq_{P\times T})$.
    \item If $P$ is $P_2$ and $T$ is $P_2$, $P_3$ or $P_5$, then player \textit{B} has the winning strategy on $(P;\leq_P)$, $(T;\leq_T)$, and $ (P\times T;\leq_{P\times T})$.
    \item If $P $ is $P_3$ and $T$ is $P_2$, then player \textit{B} has the winning strategy on $(P;\leq_P)$ and $(T;\leq_T)$, but $(1,1)$ is a winning move for player \textit{A} on $(P\times T;\leq_{P\times T})$. And the same holds if $P$ is $P_3$ and $T$ is $P_3$ or if $P$ is $P_5$ and $T$ is $P_5$: player \textit{A} can win with the move $(1,3)$ in the first case and with $(1,1)$ in the second.
\end{enumerate}
\end{example}

From now on we will work with ``powers'' of posets:
\begin{definition}\label{defpow}
    Let $(P;\leq_P)$ be a poset with global minimum $0$. We define $(P^\sigma;\leq_{P^\sigma})$ as follows:
    \begin{itemize}
        \item If $\sigma=0$, $P^0=\{0\}$ and if $\sigma=1$, $(P^1;\leq_{P^1})=(P;\leq_P)$.
        \item If $\sigma=\delta+1>1$, $(P^\sigma;\leq_{P^\sigma})=(P\times P^\delta;\leq_{P\times P^\delta})$.
        \item If $\sigma$ is a limit ordinal, $P^\sigma=\bigcup_{\alpha\in\sigma}P^{\alpha+1}\backslash(\{0\}\times P^\alpha)$. The order is defined as expected: for all $p,q\in P^\sigma$, if $p,q\in P^\alpha$ for some $\alpha\in\sigma$ and $p\leq_{P^\alpha}q $, then $p\leq_{P^\sigma}q$; if $p\in P^\alpha$ and $q\in P^\beta$ with $\alpha\in\beta\in\sigma$, then $p\leq_{P^\sigma}q$.
    \end{itemize}
\end{definition}
\begin{proposition}\label{profini}
    If \((P; \leq_P)\) is a well partial order with a global minimum $0$, then, for any ordinal \(\sigma\), \((P^\sigma; \leq_{P^\sigma})\) is a well partial order and contains a global minimum.
\end{proposition}
\begin{proof}
    By transfinite induction. Cases $\sigma=0$ and $\sigma=1$ are trivial. Let us assume that \((P^\sigma; \leq_{P^\sigma})\) contains no infinite antichains but \((P^{\sigma+1}; \leq_{P^{\sigma+1}})\) contains an infinite antichain $A=\{(p_i,q_j)\mid i,j\in\omega \}$. Assume that there are infinitely many distinct \(p_i\) in \(A\), but then \(\{p_i \mid i \in \omega\} \subset P\) would be an infinite antichain. Thus, there must be only finitely many \(p_i\) in \(A\). We may fix one of them, say \(p_k\), such that \(A_k = \{(p_k, q_j) \mid (p_k, q_j) \in A\}\) is an infinite antichain. However, \(\{p_k\} \times P^\sigma\), with the order inherited from \((P^{\sigma+1}; \leq_{P^{\sigma+1}})\), is isomorphic to \((P^{\sigma}; \leq_{P^{\sigma}})\), so there cannot be infinitely many distinct \(q_j\) either. That is, \(A\) cannot be an infinite antichain. If \(\sigma\) is a limit ordinal and for no \(\alpha \in \sigma\) does \(P^\alpha\) contain infinite antichains, then any infinite antichain \(A \subseteq P^\sigma\) must contain elements \(p\) and \(q\) such that \(p \in P^\alpha\) and \(q \in P^\beta\) with \(\alpha \in \beta \in \sigma\). But then \(p\) and \(q\) would be comparable. Similarly, we can show that \(P^\sigma\) contains no infinite descending chains. Therefore, \(P^\sigma\) is a well partial order. The global minimum can also be obtained by transfinite induction: in \(P^{\sigma+1}\) the global minimum is \((0, \min(P^\sigma))\), and if \(\sigma\) is a limit ordinal, it is \(0\) itself.
\end{proof}

From now on we may study the classes of structures \(\{(P^\sigma; \leq_{P^\sigma}) \mid \sigma \in \mathrm{Ord}\}\) determined by a well partial order \((P; \leq_{P})\) with a global minimum. The following are fundamental observations:
\begin{lemma}\label{downtrans}
    Let \((P^\sigma; \leq_{P^\sigma})\) be a well partial order with a global minimum. If player $A$ has a winning strategy on $(P^\alpha;\leq_{P^\alpha})$, then $A$ also has a winning strategy on every poset $(P^\sigma;\leq_{P^\sigma})$ with $\sigma>\alpha$. Hence, if player $B$ has a winning strategy on $(P^\sigma;\leq_{P^\sigma})$, player $B$ also has a winning strategy on $(P^\alpha;\leq_{P^\alpha})$ for every $\alpha<\sigma$.

    If $\delta$ is a limit ordinal and $B$ has a winning strategy on $(P^\alpha;\leq_{P^\alpha})$ for unboundedly many $\alpha$ less than $\delta$, then $B$ has a winning strategy on $(P^\delta;\leq_{P^\delta})$.
\end{lemma}\begin{proof}
    The first statement is proved as in \hyperref[propsen]{Proposition \ref{propsen}}. The consequence is immediate and follows from Zermelo's theorem.
    
    For the second part, note that, then, $B$ has a winning strategy on $(P^\alpha; \leq_{P^\alpha})$ for every $\alpha < \delta$. Player $A$'s first move must necessarily be some $p\in P^\gamma$ for some $\gamma < \delta$, so it will be a move in $P^{\gamma}$, where $B$ has a winning strategy.
\end{proof}

We can exploit this idea to work with families of classes of posets:
\begin{proposition}
Let $S$ be a set of well partial orders with global minima. There exists an ordinal $\xi$ such that, for every $(P;\leq_P)\in S $, player $B$ has a winning strategy on $(P^\sigma;\leq_{P^\sigma})$ for every ordinal $\sigma$ iff $B$ has a winning strategy on $(P^\xi;\leq_{P^\xi})$.
\end{proposition}
\begin{proof}
Let us define $$W=\{(P;\leq_P)\in S \mid A\text{ has a winning strategy on $(P^\sigma;\leq_{P^\sigma})$ for some $\sigma$} \}.$$Take $\xi=\bigcup_{(P;\leq_P)\in W} \alpha_P$, where $\alpha_P$ denotes the least ordinal $\gamma$ such that $A$ has a winning strategy on $(P^{\gamma};\leq_{P^\gamma})$. Clearly, if $B$ has a winning strategy on $(P^{\xi};\leq_{P^\xi}) $, then $ (P;\leq_{P})\not\in W$, i.e., $B$ has a winning strategy on $(P^\sigma;\leq_{P^\sigma})$ for every $\sigma$.
\end{proof}For example, \(S\) could be the set of all such posets with cardinality less than or equal to a fixed cardinal \(\kappa\). 

The argument in the preceding proposition, as well as phenomena such as the one described in \hyperref[downtrans]{Lemma \ref{downtrans}}, are well known in model theory. In that context, the ordinal $\xi$ would correspond to a ``Hanf number'' for the property ``$B$ has a winning strategy'' (see \cite[Ch. 4]{Baldwin2009}).
\begin{notation}
    \hyperref[downtrans]{Lemma \ref{downtrans}} shows that if $B$ does not have a winning strategy on all the posets $(P^\sigma;\leq_{P^\sigma})$, there must be an ordinal $\xi$ such that $B$ has a winning strategy on $(P^\xi;\leq_{P^\xi})$ but not on $(P^{\gamma };\leq_{P^{\gamma}})$ for all $\gamma\geq\xi+1$. We let $\textnormal{ch}(P;\leq_{P})$ denote $\xi+1$ and, for a set $S$ of posets, $\text{ch}(S)=\bigcup_{(P;\leq_{P})\in S}\textnormal{ch}(P;\leq_{P})$.
\end{notation}
\subsection{Algebraic approach} We shall see in this section that the posets given by \hyperref[lemposetstr]{Lemma \ref{lemposetstr}} fit perfectly into the classes of posets under consideration. We first require a simple preliminary lemma.
\begin{lemma}
    Let $\Gamma\not\subset\omega$ be a set of ordinals, and define\[\varepsilon :=\begin{cases}\left(\bigcup_{\gamma\in\Gamma}\overline{e}(\overline{e}(\gamma))\right) + 1, & \text{if } \bigcup_{\gamma\in\Gamma}\overline{e}(\overline{e}(\gamma)) \in \{\overline{e}(\overline{e}(\gamma)) \mid \gamma \in \Gamma\}, \\\bigcup_{\gamma\in\Gamma}\overline{e}(\overline{e}(\gamma)), & \text{otherwise}.\end{cases}\] Then, every ordinal of the form \(\omega^{\alpha}\), with $\alpha=\omega^{\epsilon_1} e_1 + \cdots + \omega^{\epsilon_n} e_n + \omega^{\varepsilon} e$ (in Cantor normal form), is closed under $\otimes$ and $\oplus$ with elements of $\Gamma$; that is,\[\alpha_1 \otimes \gamma_1 \oplus \cdots \oplus \alpha_k \otimes \gamma_k \in \omega^\alpha\]for arbitrary $\alpha_1, \dots, \alpha_k \in \omega^\alpha$ and $\gamma_1, \dots, \gamma_k \in \Gamma$. Ordinals of this form will be called \emph{$\Gamma$-closed}. Moreover, if $\omega^\alpha$ is $\Gamma$-closed, then every element $\beta\in\langle \Gamma \rangle^\alpha$ can be written as \[ \beta = \omega^{\alpha_1} \otimes \beta_1 + \cdots + \omega^{\alpha_m} \otimes \beta_m + \beta_{m+1}, \] where the $\alpha_i$'s are $\Gamma$-closed, $\alpha>\alpha_1 > \cdots > \alpha_m$, and $\beta_1, \dots, \beta_{m+1} \in \langle \Gamma \rangle^{\omega^\varepsilon}$.
    
    If $\Gamma\subset\omega$, any ordinal of the form $\omega^\alpha$ is $\Gamma$-closed and the corresponding decomposition was written in (\ref{ecdes}).
\end{lemma}
\begin{proof}
The first part follows from \hyperref[lemamon]{Lemma \ref{lemamon}}; in particular, note that $\omega^{\omega^\varepsilon}$ is the least ordinal closed under $\oplus$ and $\otimes$ that contains $\Gamma$.

For the second part, we may factor out in $\beta$ common terms of the form $\omega^{\alpha_i}$ that are $\Gamma$-closed and chosen to be as large as possible. In this way, the coefficients multiplying each $\omega^{\alpha_i}$ are elements of $\langle \Gamma \rangle^{\omega^\varepsilon}$. Finally, we may replace $\oplus$ by $+$, since in Cantor normal form the largest term of $\omega^{\alpha_{i+1}} \otimes \beta_{i+1}$ has an exponent smaller than the exponent of the smallest term of $\omega^{\alpha_i} \otimes \beta_i$.

And the case $\Gamma \subset \omega$ had already been considered.
\end{proof}
\begin{proposition}
    Let $\Gamma $ be a set of ordinals, let $\varepsilon$ be defined as in the previous lemma, and denote \((P;\leq_P) = (\mathcal{S}^{\omega^\varepsilon}; \leq_{\mathcal{S}^{\omega^\varepsilon}}).\) Then, for every poset $(P^\sigma; \leq_{P^\sigma})$ there exists a $\Gamma$-closed ordinal $\omega^{\sigma'}$ such that \((P^\sigma; \leq_{P^\sigma})\cong(\mathcal{S}^{\sigma'}; \leq_{\mathcal{S}^{\sigma'}}).\)
\end{proposition}
\begin{proof}
    By transfinite induction. If \((P^\sigma; \leq_{P^\sigma})\cong(\mathcal{S}^{\sigma'}; \leq_{\mathcal{S}^{\sigma'}})\) and $f^\sigma$ is an isomorphism, the poset $(P^{\sigma+1}; \leq_{P^{\sigma+1}})$ is isomorphic to $(\mathcal{S}^{\sigma'+\omega^\varepsilon}; \leq_{\mathcal{S}^{\sigma'+\omega^\varepsilon}})$ and the function $f^{\sigma+1}$ given by $f^{\sigma+1}\left((\alpha,\beta)\right)=\omega^{\sigma'}\otimes\alpha+f^\sigma(\beta)$ is an isomorphism. If $\sigma$ is a limit ordinal, take $\sigma'=\bigcup_{\alpha\in\sigma}\alpha'$ and $f^{\sigma}=\bigcup_{\alpha\in\sigma}f^{\alpha+1}\upharpoonright(P^{\alpha+1}\backslash\{0\}\times P^\alpha)$ is an isomorphism. And, in particular, if $\Gamma\subset\omega$,  then $\sigma'=\sigma$ and we can define isomorphisms similarly.
\end{proof}

Nevertheless, observe that, in this context, nothing prevents us from considering a broader class of posets. Instead of studying the class of powers of a poset given by \hyperref[defpow]{Definition \ref{defpow}}, using an appropriate set $\Gamma$ of ordinals (see \hyperref[defwellor]{Definition \ref{defwellor}}), we may work with the corresponding class \( \{(\mathcal{S}^\sigma;\leq_{\mathcal{S}^\sigma}) \mid \Gamma \subset \omega^\sigma\},\) which is more general. It is straightforward to verify that \hyperref[downtrans]{Lemma \ref{downtrans}} also holds for this class of posets.
\begin{notation}
    Given an appropriate set $\Gamma$ of ordinals, if $B$ has not the winning strategy on all the posets $(\mathcal{S}^\sigma;\leq_{\mathcal{S}^\sigma})$, we let $\textnormal{ch}(\Gamma)$ denote the ordinal $\xi+1$ such that player $B $ has the winning strategy on $(\mathcal{S}^\xi;\leq_{\mathcal{S}^\xi})$ but not on $(\mathcal{S}^\gamma;\leq_{\mathcal{S}^\gamma}) $ for all $\gamma\geq\xi+1$.
\end{notation}
\subsection{Well partial orders} In this section we address a more practical question: the characterization of those sets of ordinals that give rise to well partial orders. Arguments such as those presented in \hyperref[profini]{Proposition \ref{profini}} show that, if $\Gamma \subset \omega$, then any poset $(\mathcal{S}^\sigma; \leq_{\mathcal{S}^\sigma})$ is a well partial order and, therefore, the game of Chomp can be played on these posets.
\begin{figure}[h!]
\def\svgwidth{0.87\columnwidth}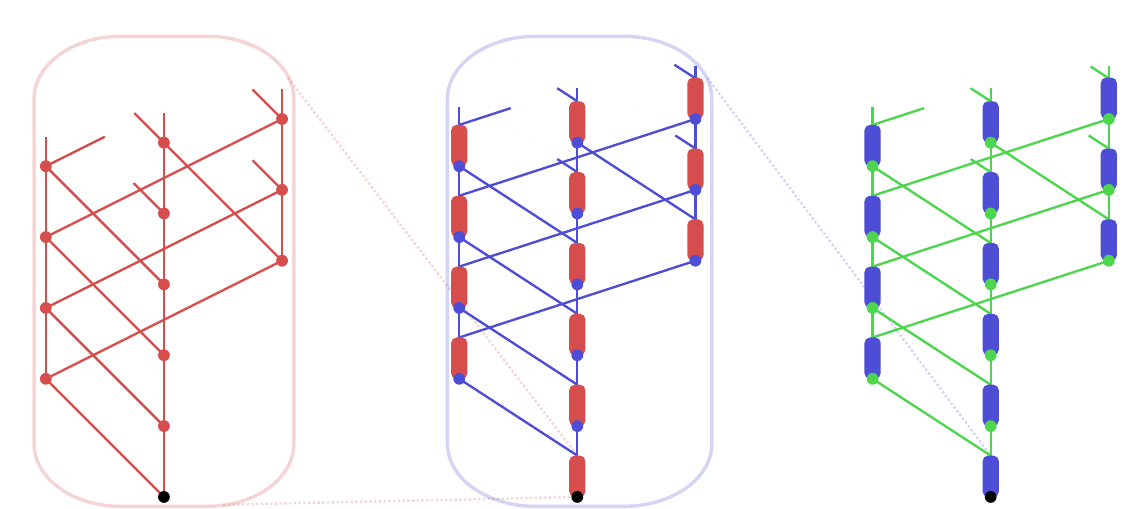    \caption{Posets $(\mathcal{S}^\sigma;\leq_{\mathcal{S^\sigma}})$ with $\mathcal{S}^\sigma=\langle3,5\rangle^\sigma$ and $\sigma\in\{1,2,3\}$.}\label{fig1}
\end{figure}

This, however, fails in general once we move to more general sets of generators:
\begin{example}
    The sets $$\Gamma_1=\{\omega+1\},\qquad \Gamma_2=\{\omega^n+1 \mid n\in\omega\backslash\{0\} \},\qquad \Gamma_3=\{2,3, \omega^2+\omega+1 \} $$generate the following infinite antichains $$A_1=\{\omega\cdot n+n \mid n\in\omega\backslash\{0\} \},\qquad A_2=\Gamma_2,\qquad A_3=\{\omega^2+\omega\cdot(1+2n)+1 \mid n\in\omega\}. $$
\end{example}
\begin{definition}\label{defwellor}
    A set $\Gamma$ of ordinals will be called a \textit{well partial order generator} ($\mathsf{wpog}$ for short) if the poset structure on $\langle \Gamma\rangle^\sigma$ is a well partial order for every $\sigma$ such that $\Gamma\subset\omega^\sigma$.
\end{definition}
The following proposition gives us some information on how these sets can be. In particular, we obtain that in any finite $\mathsf{wpog}$ the smallest generators must be of the form $\omega^\gamma p_1,\dots,\omega^{\gamma}p_n$.
\begin{proposition}
Let $\Gamma$ be a finite $\mathsf{wpog}$. If $\alpha_1, \dots, \alpha_m$ are all the elements of $\Gamma$ of the form
\(
\alpha_i = \omega^{\gamma_1} a_{i1} + \cdots + \omega^{\gamma_n} a_{in}
\)
with $a_{i1} \neq 0$, and they are the smallest elements of $\Gamma$, then $a_{ij} = 0$ for all $i$ and $j \neq 1$.
\end{proposition}
\begin{proof}
    Otherwise, let us consider the quotients $a_{in} / a_{i1}$ (although the argument works if we take any $1<k \leq n$ instead of $n$). Now, let us take the largest of these quotients and, for simplicity, assume that it is $a_{1n} / a_{11}$. Then $\{\alpha_1\otimes p \mid p\in\omega\backslash\{0\}\} $ would be an infinite antichain. To see this, let us take $\alpha_1 \otimes q$ and $\alpha_1 \otimes p$ with $0 < p \lneq q$. We will show that they are not comparable.
    
    If they are indeed comparable, $\alpha_1 \otimes q-\alpha_1 \otimes p$ is generated by $\Gamma$. Since $\alpha_1, \dots, \alpha_m$ are the smallest elements of $\Gamma$, we can obtain $\alpha_1 \otimes q-\alpha_1 \otimes p$ only through a linear combination of the form $\alpha_1\otimes b_1\oplus\cdots\oplus\alpha_m\otimes b_m$ with $b_1,\dots,b_m\in\omega$. This gives us $n$ equations, two of which are $b_1 a_{11}+\cdots+b_m a_{m1}=a_{11}(q-p)$ and $b_1 a_{1n}+\cdots+b_m a_{mn}=a_{1n} q$. If $m = 1$, there is a solution only if $p = 0$, which is absurd. If $m\geq1$, we obtain $$b_2=\frac{a_{11}a_{1n}p+\sum_{i=3}^mb_i(a_{1n}a_{i1}-a_{in}a_{11})}{a_{11}a_{2n}-a_{1n}a_{21}}. $$ Because of the way the quotients $a_{in} / a_{i1}$ are, we see that the numerator is always positive, while the denominator is negative, so $b_2$ would be a negative integer, which is absurd.

    And if all the quotients $a_{in} / a_{i1}$ are equal, the system of equations does not even have a solution. This proves that $\{\alpha_1\otimes p \mid p\in\omega\backslash\{0\}\} $ is an infinite antichain.
\end{proof}
Thus, we do not lose much by assuming the condition of the following result.
\begin{theorem}\label{charwpog}
    Let $\Gamma$ be a finite set of ordinals containing a set $\{p_1,\dots,p_n\}$ of relatively prime natural numbers. If $(\mathcal{S}^{\overline{e}(\Gamma)+1};\leq_{\mathcal{S}^{\overline{e}(\Gamma)+1}})$ contains no infinite antichains, then $\Gamma$ is a $\mathsf{wpog}$.
\end{theorem}
\begin{proof}
    By transfinite induction on the exponent of the structures $(\mathcal{S}^{\sigma};\leq_{\mathcal{S}^{\sigma}})$. The base case is our hypothesis. Now, assume that $\Gamma\subset\omega^{\sigma}$ and that there are no infinite antichains in $(\mathcal{S}^{\sigma};\leq_{\mathcal{S}^{\sigma}})$. For any $\zeta \in \langle \Gamma \rangle^{\sigma + 1}$, either $\zeta\in\langle\Gamma\rangle^{\sigma}$ or there exist $\alpha \in \langle \Gamma \rangle^{\overline{e}(\Gamma) + 1}$ and $\xi \in \langle \Gamma \rangle^{\sigma}$ such that $\zeta = \omega^{\gamma} \otimes \alpha \oplus \xi$ for some $\gamma $ satisfying $\overline{e}(\omega^\gamma\otimes\alpha)=\sigma $. Let us consider an infinite set \( A= \{\alpha_i \otimes \omega^{\gamma} \oplus \xi_j \mid i,j \in \omega\}\subset\langle\Gamma\rangle^{\sigma+1}\), where $\overline{e}(\omega^\gamma\otimes\alpha_i)=\sigma $ and $\{\alpha_i \mid i\in\omega\}$ and $\{\xi_j \mid j\in\omega\}$ are arbitrary subsets of $\langle \Gamma \rangle^{\overline{e}(\Gamma)+1}$ and $\langle \Gamma \rangle^{\sigma}$, respectively. If for the terms with $\omega^\sigma$ of $A$ there are infinitely many distinct coefficients, then the set $\{\alpha_i \mid i\in\omega\}$ must be infinite. Hence we may find $i,i'\in\omega$ such that $\alpha' = \alpha_{i'} - \alpha_i \in \langle \Gamma \rangle^{\overline{e}(\Gamma)+1}$ and $\overline{e}(\alpha') = \overline{e}(\alpha_{i'})$. Consequently, \[ \alpha_{i'} \otimes \omega^{\gamma} \oplus \xi_{j'} -\alpha_i \otimes \omega^{\gamma} \oplus \xi_j = \alpha' \otimes \omega^{\gamma} \oplus \xi_{j'}, \] for arbitrary $j,j'$. Therefore, if there are infinitely many coefficients for the $\omega^\sigma$ term, the set $A$ cannot be an infinite antichain. Before proceeding, we will need the following claim.
\begin{claim}
    Let $\alpha=\omega^{\alpha_1} a_1 + \cdots + \omega^{\alpha_m} a_m\in \langle \Gamma \rangle^{\overline{e}(\Gamma)+1}$. For every ordinal $\beta$ satisfying $\alpha_m\leq\beta<\alpha_1 $, there exists $k_{\beta}\in \omega$ such that \[\alpha_\beta:=\omega^{\beta}k_{\beta}  + \omega^{\alpha_{i-1}} a_{i-1}+ \cdots + \omega^{\alpha_m} a_m \in\langle \Gamma \rangle^{\overline{e}(\Gamma)+1}, \] where $\alpha_{i-1}<\beta\leq\alpha_i$. If $\alpha_i-\alpha_{i-1}\geq\omega$ and $\alpha_{i-1}<\beta<\alpha_i$, we can take $k_\beta=0$.
\end{claim}
\begin{proof}[Proof of the claim] This is clear for $\alpha_m$. For $\alpha_m<\beta<\alpha_1$, consider the set $ \{\alpha\oplus\omega^{\beta}\otimes(b_1p_1+\cdots+b_np_n) \ |\ b_1,\dots,b_n\in\omega\}\subset\langle\Gamma \rangle^{\overline{e}(\Gamma)+1},$ which is not an antichain, by assumption. Therefore, we can find two comparable elements with difference $\omega^{\beta}d_\beta+\omega^{\alpha_{i-1}} a_{i-1} + \cdots + \omega^{\alpha_m} a_m\in\langle\Gamma \rangle^{\overline{e}(\Gamma)+1}$. Then, we take $k_{\beta} = d_\beta$.

For the second part, consider the set $\{\alpha \oplus \omega^{\alpha_{i-1}+k} p_1 \mid k \in \omega\}$. By similar arguments, we see that there are infinitely many elements of the form $\omega^{\alpha_{i-1}+k} p_1+ \omega^{\alpha_{i-1}} a_{i-1}+ \cdots + \omega^{\alpha_m} a_m $ in $\langle \Gamma \rangle^{\overline{e}(\Gamma)+1}$. Then, if we were unable to obtain $\omega^{\alpha_{i-1}} a_{i-1} + \cdots + \omega^{\alpha_m} a_m$, then there would have to be infinitely many generators in $\Gamma$, which is not the case, so $\omega^{\alpha_{i-1}} a_{i-1} + \cdots + \omega^{\alpha_m} a_m\in \langle \Gamma \rangle^{\overline{e}(\Gamma)+1}$.
\end{proof}
    From what we saw at the beginning, an antichain in $\langle\Gamma\rangle^{\sigma+1}$ would give us an antichain of the form $A=\{\alpha \otimes \omega^{\gamma} \oplus \xi_j \mid j \in \omega\}$ where $\Xi=\{\xi_j \mid j\in\omega\}$ is some subset of $\langle \Gamma \rangle^{\sigma } $ and $\alpha\in\langle \Gamma \rangle^{\overline{e}(\Gamma)+1}$ is fixed and satisfies $ \overline{e}(\omega^\gamma\otimes\alpha)=\sigma$. Assume that $\beta=\overline{e}(\Xi)\in\{\overline{e}(\xi_j)\mid j\in\omega\}$; that is,  for all $j\in\omega$, the first unfixed term of $\alpha \otimes \omega^{\gamma} \oplus \xi_j$ is of the form $\omega^\beta \cdot n_j^\beta$ whenever $\alpha_{i-1} < \beta < \alpha_i$, or $\omega^\beta \cdot (n_j^\beta + a_i)$ when $\beta = \alpha_i$. If there are infinitely many $n_j^\beta$, we can apply the claim and obtain $\{\alpha_\beta \otimes \omega^{\gamma} \oplus \xi_j \mid j \in \omega\}\subset\langle\Gamma\rangle^\sigma$. This cannot be an antichain, so we may compare two elements $\alpha_\beta \otimes \omega^{\gamma} \oplus \xi_j$ and $\alpha_\beta \otimes \omega^{\gamma} \oplus \xi_{j'}$. It is not hard to see that, consequently, we can also compare $\alpha \otimes \omega^{\gamma} \oplus \xi_j$ and $\alpha \otimes \omega^{\gamma} \oplus \xi_{j'}$, which is absurd. Therefore, we may assume that there are only finitely many distinct $n_j^\beta$’s and, in particular, we may fix one of them and obtain a new infinite antichain $A' \subseteq A$ whose first and second coefficients are fixed. If $\overline{e}(\Xi)>\overline{e}(\xi_j)$ for all $j\in\omega$, we apply the second part of the claim in a similar way.
    
    We may continue fixing coefficients, using the second part of the claim in the cases where $\alpha_i - \alpha_{i-1} \geq \omega$. After a finite number of steps, we will have fixed all coefficients between $\alpha_1$ and $\alpha_m$, and we will conclude that there is an infinite subset of $\{\xi_j \mid j \in \omega\}$ that forms an antichain, which is impossible. Hence, $A$ cannot be an antichain. Finally, if $\sigma$ is a limit ordinal, we proceed as in \hyperref[profini]{Proposition \ref{profini}}.
\end{proof}
It is worth noting that this theorem essentially shows that there exists a ``Hanf number'' for the property ``the poset contains no antichains'' and that, in fact, this number is the smallest
possible ordinal. This, in part, motivates \hyperref[ques]{Question \ref{ques}}.

\begin{corollary}
    Let $\Gamma$ be a finite set of generators, and let $ p_1, \dots,p_n \in \Gamma$ be relatively prime. If the remaining elements of $\Gamma$ are $\omega^{\alpha+1} a_1 + \omega^{\alpha} b_1, \dots, \omega^{\alpha+1} a_m + \omega^{\alpha} b_m$, then $\Gamma$ is a $\mathsf{wpog}$.
\end{corollary}
\begin{proof}
    Combining \hyperref[charwpog]{Theorem \ref{charwpog}} with the fact that $p_1, \dots,p_n$ do not generate infinite antichains, we see that we only need to look for antichains between $\omega^{\alpha+1}$ and $\omega^{\alpha+2}$. Let $A \subset \langle \Gamma \rangle^{\alpha+2}$ be an infinite antichain; then, its elements can be written as \[ \omega^{\alpha+1}(p_1 r_1 + \cdots + p_n r_n)+\omega^{\alpha+1}(a_1 s_1 + \cdots + a_m s_m)+\omega^{\alpha}(b_1 s_1 + \cdots + b_m s_m)+\beta,\]where $\beta \in \langle p_1,\dots,p_n \rangle^{\alpha+1}$ and $r_1,\dots,r_n, s_1,\dots,s_m \in \omega$. For simplicity, we will denote this by \( \omega^{\alpha+1}\overline{p}\,\overline{r}+\omega^{\alpha+1}\overline{a}\,\overline{s}+\omega^{\alpha}\overline{b}\,\overline{s}+\beta\). If there are infinitely many elements in $A$ with distinct $\omega^{\alpha+1}\overline{p}\,\overline{r}$ terms, since  $S=\langle p_1,\dots,p_n\rangle^{1}$  generates all natural numbers above its Frobenius number, we may find two elements of $A$, \[ \zeta'=\omega^{\alpha+1}\overline{p}\,\overline{r}'+ \omega^{\alpha+1}\overline{a}\,\overline{s}' + \omega^{\alpha}\overline{b}\,\overline{s}' + \beta'\quad\text{and}\quad\zeta=\omega^{\alpha+1}\overline{p}\,\overline{r}+\omega^{\alpha+1}\overline{a}\,\overline{s}+ \omega^{\alpha}\overline{b}\,\overline{s}+\beta,\]such that \(\overline{p}\,\overline{r}+\overline{a}\,\overline{s} \in S\) and $ \overline{p}\,\overline{r}+\overline{a}\,\overline{s}\leq_{S} \overline{p}\,\overline{r}'$. But then, \[ \zeta' - \zeta=\omega^{\alpha+1}(\overline{p}\,\overline{r}'-(\overline{p}\,\overline{r}+\overline{a}\,\overline{s}))+ \omega^{\alpha+1}\overline{a}\,\overline{s}'+ \omega^{\alpha}\overline{b}\,\overline{s}'+ \beta'\in \langle \Gamma \rangle^{\alpha+2},\]that is, $\zeta$ and $\zeta'$ are comparable. Therefore, there are only finitely many elements in $A$ with distinct $\omega^{\alpha+1}\overline{p}\,\overline{r}$ terms. We may fix one of them and obtain a new infinite antichain $A' \subseteq A$.

    So, we now suppose that in $A'$ there are infinitely many elements with distinct $\omega^{\alpha+1}\overline{a}\,\overline{s} + \omega^{\alpha}\overline{b}\,\overline{s}$ terms. Therefore, at some point we will pass the Frobenius number of $\{p_1,\dots,p_n\}$, and there will be infinitely many terms such that $\omega^{\alpha+1}\overline{a}\,\overline{s}, \omega^{\alpha}\overline{b}\,\overline{s} \in \langle p_1,\dots,p_n \rangle^{\alpha+2}$. Thus, there is an infinite subset of $A'$ contained in $\langle p_1,\dots,p_n \rangle^{\alpha+2}$, which contains no antichains. Therefore, we could also fix a term of the form $\omega^{\alpha+1}\overline{a}\,\overline{s} + \omega^{\alpha}\overline{b}\,\overline{s}$ and obtain a new infinite chain $A''\subseteq A'$; however, this would give us an infinite antichain in $\langle p_1,\dots,p_n \rangle^{\alpha+1}$, which is impossible.
\end{proof}
This shows that the class of well partial order generators is quite diverse.

\section{Bounds for the function $\textnormal{ch}(\Gamma)$}
We begin by proving, constructively, that $\textnormal{ch}(\Gamma)=1$ for a special family of generators, which we now proceed to define.
\begin{definition}
    Given a minimal set $\{p_1,\dots,p_n\}\subset\omega$ of generators, we will say that $\langle p_1,\dots,p_n\rangle^1\subseteq\omega$ has \textit{maximal embedding dimension} if $\min\{p_1,\dots,p_n\}=n$.
\end{definition}
\begin{proposition}\label{propmaxemd}
    If $\mathcal{S}^1 $ has maximal embedding dimension and $B$ has a winning strategy on $(\mathcal{S}^1;\leq_{\mathcal{S}^1})$, then $B$ has a winning strategy on $(\mathcal{S}^\sigma;\leq_{\mathcal{S}^\sigma})$ for all $\sigma\in\textnormal{Ord}$.
\end{proposition}\begin{proof}
    By the second part of \hyperref[downtrans]{Lemma \ref{downtrans}}, we only need to prove the proposition for successors. Then, take $\sigma=\delta+1$ and player $B$'s winning strategy will consist of ``forcing'' player $A$ to play as if they were playing on $\mathcal{S}^1$. We will follow the same steps described in \cite[Theorem 3.4]{ALCO_2018__1_3_371_0}, but adapting each move to our case.

    If the minimal set of generators is $\{p_1,\dots,p_n\}$, we divide $\mathcal{S}^\sigma$ into layers $\{\mathcal{S}_k\}_{k\in\mathbb{N}}$ in the following way: $\mathcal{S}_0:=\{0\}$ and $\mathcal{S}_k:=\bigcup\{\mathcal{S}^\sigma_{(k-1)n+p_i} \mid 1\leq i\leq n \}$. For any first move by $A $ in $\mathcal{S}_k $, we will see that $B$ can respond with a move such that the remaining poset $P$ satisfies two conditions: $\bigcup_{l<k}\mathcal{S}_l\subseteq P\subsetneq\bigcup_{l\leq k}\mathcal{S}_l$ and $A$ will eventually have to be the first to choose an element from some layer $\mathcal{S}_l$ with $l<k$. We need to distinguish two cases:
    \begin{itemize}
    \item[\textit{i})] Assume that player $A$ chooses an element of the form $ \omega^\delta\cdot kn+\beta $. By \cite[Lemma 3.3]{ALCO_2018__1_3_371_0}, the first condition is satisfied. Now, note that, by \cite[Theorem 3.4]{ALCO_2018__1_3_371_0}, the number of generators is even, and this gives us the second condition, as we will be able to divide any layer into two branches, just like in \hyperref[propsim]{Proposition \ref{propsim}}, and follow the strategy of imitating the other player's moves. For example, player $B$'s next move will be of the form $\omega^\delta\cdot ((k-1)n+p_i)+\beta$ for some $2\leq i\leq n$.
    \item[\textit{ii})] Now assume that player $A$ chooses an element of the form $ \omega^\delta\cdot ((k-1)n+p_i)+\beta $ for some $2\leq i\leq n$. Then $B$ picks $ \omega^\delta\cdot kn+\beta $ and the resulting poset is exactly the same as if $A$ had started with that move of $B$ and $B$ had responded with $\omega^\delta\cdot ((k-1)n+p_i)+\beta$. This is the same situation we described at the end of the previous case, so we already have the two properties we are interested in.
    \end{itemize}Finally, by iterating this strategy, $A$ will be forced to be the first to choose an element from $\mathcal{S}_0$, which only contains $0$, so this player will have lost.
\end{proof}
For our main result, we will need the following two lemmas.
\begin{lemma}\label{sumprodef}
    Natural sum and product are $\Sigma_1$ functions.
\end{lemma}
\begin{proof}
The relation $\sqsubset\subseteq\text{Ord}\times\text{Ord}$ given by $$(\gamma,\delta)\sqsubset(\alpha,\beta)\iff (\gamma<\alpha\wedge\delta<\beta)\vee(\gamma=\alpha\wedge\delta<\beta)\vee(\gamma<\alpha\wedge\delta=\beta) $$is well-founded (we only need to use that the usual order in the ordinals is well-founded). Now, we can define, for every pair $(\alpha,\beta)\in\text{Ord}\times\text{Ord}$ and function $f:\{(\gamma,\delta) \mid (\gamma,\delta)\sqsubset(\alpha,\beta)\}\to V$, the $\Sigma_1$ function $G\big((\alpha,\beta),f\big)$: $$G\big((\alpha,\beta),f\big)=\bigcup\{f((\gamma,\delta))+1 \mid (\gamma,\delta)\sqsubset (\alpha,\beta)\}.$$ By $\Sigma_1$-recursion on $\sqsubset$ (see \cite[Theorem 3.1]{drake1974set}), this defines a unique $\Sigma_1$ function $F:\text{Ord}\times\text{Ord}\to V$ such that $F(x)=G(x,F\upharpoonright\{y:y\sqsubset x\})$. By \cite[Theorem 2.4]{altman}, $F$ is precisely $\oplus$. Similarly, the $\Sigma_1$ function $H\big((\alpha,\beta),f\big)$ given by $$H\big((\alpha,\beta),f\big)=\bigcap\left\{\varepsilon \mid (\forall \gamma\in\alpha)(\forall \delta\in\beta)\big( \varepsilon\oplus f((\gamma,\delta))>f((\alpha,\delta))\oplus f((\gamma,\beta))\big)\right\} $$defines a unique $\Sigma_1$ function which is $\otimes$.
\end{proof}

\begin{lemma}\label{lemmadef}
    For every finie set $\{\gamma_1,\dots,\gamma_n\}=\Gamma\subset\textnormal{Ord}$, the correspoding class $\{ (\mathcal{S}^\sigma;\leq_{\mathcal{S}^\sigma}) \mid\text{ $\omega^\sigma$ is $\Gamma$-closed} \}$ is $\Sigma_1$-definable with $\gamma_1,\dots,\gamma_n$ as parameters.
\end{lemma}\begin{proof} The sets $\langle\Gamma\rangle^\sigma$ are defined by
    \begin{gather*}
    \varphi(x)=(\forall \alpha\in x)\exists\alpha_1\dots\exists\alpha_n\big( \alpha=\alpha_1\otimes\gamma_1\oplus\cdots\oplus\alpha_n\otimes\gamma_n \big) \wedge\\\wedge\exists \delta\Big(\delta=\text{rank}(x)\wedge(\forall \beta_1\in\delta)\dots(\forall \beta_n\in\delta)\big(\beta_1\otimes\gamma_1\oplus\cdots\oplus\beta_n\otimes\gamma_n\in x\big)\Big).
\end{gather*} By the previous lemma and \cite[Corollary 3.2]{drake1974set}, $\varphi(x)$ is a $\Sigma_1$ formula.

For the relation $\leq_{\mathcal{S}^\sigma}$ we have $$\psi(x,y)=\exists z(\varphi(z)\wedge x\in z\wedge y\in z\wedge x\in y\wedge\text{ot}(y\backslash x)\in z)$$ (note that $\text{ot}$ is also a $\Sigma_1$ function), which proves that $\psi(x,y)$ is a $\Sigma_1$ formula.
\end{proof}

\begin{theorem}\label{theoomeg1}
    Let $\Gamma$ be a finite $\mathsf{wpog}$. If $B$ does not have a winning strategy for every poset of the class $\{(\mathcal{S}^{\sigma};\leq_{\mathcal{S}^\sigma}) \mid \Gamma\subset\omega^{\sigma}\}$ generated by $\Gamma$, then $\textnormal{ch}(\Gamma)<\omega_1$ if $\Gamma\subset\omega$ and $\textnormal{ch}(\Gamma)<|\bigcup\Gamma|^+$ otherwise.
\end{theorem}
\begin{proof}
    For all $n\in\omega\backslash\{0\}$, we can define the sentence $\psi_{2n}\equiv$ ``player $B$ can avoid losing after $2n$ moves of Chomp''. For example, for $n=2$ we have \begin{gather*} \psi_4=\forall x_0(x_0=0\vee\exists x_1(\overline{x_0,x_1}\wedge x_1\neq0\wedge\\\wedge\forall x_2(\overline{x_0,x_1,x_2}\rightarrow x_2=0\vee\exists x_3(\overline{x_0,\dots,x_3}\wedge x_3\neq0)))) \end{gather*}
where $\overline{x_0,\dots,x_n}\equiv$ ``$(x_0,\dots,x_n)$ is a play of Chomp (or a fragment of it)'':
   $$ \overline{x_0,\dots,x_n}=\bigwedge_{i\leq n}\bigwedge_{j< i}x_{j}\not\leq x_{i}.$$ Then, since the Chomp game is finite, we have $$\Psi=\bigwedge_{n\in\omega\backslash\{0\}}\psi_{2n}\equiv\text{``$B$ has a winning strategy''}.$$

   By \hyperref[lemmadef]{Lemma \ref{lemmadef}}, we can apply \cite[Theorem 4.2]{Bagaria2012-BAGCN} to the class of structures $C=\{(\mathcal{S}^{\sigma};\leq_{\mathcal{S}^\sigma}) \mid \text{$\omega^\sigma$ is $\Gamma$-closed}\}$. Let $\kappa$ be equal to $\omega_1$ if $\Gamma \subset \omega$, and equal to $|\bigcup\Gamma|^+$ otherwise. Then, for every $(\mathcal{S}^\beta;\leq_{\mathcal{S}^\beta})\in C$, there exists $(\mathcal{S}^{\alpha};\leq_{\mathcal{S}^\alpha})\in C $ with $\alpha<\kappa$, such that there is some elementary embedding $j:\mathcal{S}^\alpha\to\mathcal{S}^\beta$. Take, for example, the least $\xi<\kappa$ for which there exists one of these elementary embeddings from $\mathcal{S}^\xi$ to $\mathcal{S}^\sigma$ with $\sigma\geq\kappa$. Now, if $B$ has a winning strategy on $(\mathcal{S}^\xi;\leq_{\mathcal{S}^\xi})$, since there is an elementary embedding that preserves the sentences $\psi_{2n}$, we have $(\mathcal{S}^\sigma;\leq_{\mathcal{S}^\sigma})\vDash\Psi$. Thus, by \hyperref[downtrans]{Lemma \ref{downtrans}}, $B$ has a winning strategy on every $(\mathcal{S}^\alpha;\leq_{\mathcal{S}^\alpha})$ with $\alpha<\kappa$. Now we can transfer $\Psi$ to all structures above $\sigma$ using the corresponding elementary embeddings. Then, if $B$ has a winning strategy on $(\mathcal{S}^\xi;\leq_{\mathcal{S}^\xi})$, $B$ also has a winning strategy on all posets in $ C$ and therefore, again by \hyperref[downtrans]{Lemma \ref{downtrans}}, on all posets in $\{(\mathcal{S}^{\sigma};\leq_{\mathcal{S}^\sigma}) \mid \Gamma\subset\omega^{\sigma}\}$, which proves that $\textnormal{ch}(\Gamma)<\kappa$.
\end{proof}
Although this theorem is quite general (and, in fact, its proof can be adapted for submonoids with the usual sum of ordinals instead of the natural sum), it is clear that it is far from optimal, and we believe that the bound provided can be significantly improved: very little information about the structure of the posets, which could be valuable, has been used so far. And, for families of posets with a more complex definition ($\Sigma_n$-definable for some $n \geq 2$), we would require large cardinals in the proof, which is even less optimal. Moreover, here we have focused only on posets generated by finite sets.
\begin{question}\label{ques}
Can we prove that $\textnormal{ch}(\Gamma)$ is the smallest possible ($\overline{e}(\Gamma)+1$ or $\overline{e}(\Gamma)$)?
\end{question}

If the answer to \hyperref[ques]{Question \ref{ques}} is affirmative, \cite[Theorem 6.5]{ALCO_2018__1_3_371_0} would show us that the problem of determining the result of playing Chomp on any $(\mathcal{S}^\sigma,\leq_{\mathcal{S}^\sigma})$ generated by natural numbers is decidable. However, the problem of finding explicit winning strategies in each case may still remain unresolved, leaving many open questions regarding the chomp game on these structures.

\section*{Acknowledgments}
The author wishes to acknowledge Michael Lieberman and Ignacio García-Marco for their comments during the development of this paper. In particular, Ignacio García-Marco introduced this topic to the author and proposed \hyperref[ques]{Question \ref{ques}}.

The research was supported by the Brno University of Technology from the project no. FSI-S-23-8161.

\vspace{5mm}


\begin{thebibliography}{GMK18}

\bibitem[Alt17]{altman}
Harry~J. Altman.
\newblock Intermediate arithmetic operations on ordinal numbers.
\newblock {\em Mathematical Logic Quarterly}, 63(3-4):228--242, 2017.

\bibitem[Bag12]{Bagaria2012-BAGCN}
Joan Bagaria.
\newblock C$^{(n)}$-cardinals.
\newblock {\em Archive for Mathematical Logic}, 51(3-4):213--240, 2012.

\bibitem[Bal09]{Baldwin2009}
John~T. Baldwin.
\newblock {\em Categoricity}, volume~50 of {\em Univ. Lect. Ser.}
\newblock Providence, RI: American Mathematical Society (AMS), 2009.

\bibitem[dP77]{DEJONGH1977195}
D.H.J {de Jongh} and Rohit Parikh.
\newblock Well-partial orderings and hierarchies.
\newblock {\em Indagationes Mathematicae (Proceedings)}, 80(3):195--207, 1977.

\bibitem[Dra74]{drake1974set}
F.R. Drake.
\newblock {\em Set Theory: An Introduction to Large Cardinals}.
\newblock Studies in logic and the foundations of mathematics. North-Holland Publishing Company, 1974.

\bibitem[FR15]{intro}
Stephen~A. Fenner and John Rogers.
\newblock Combinatorial game complexity: an introduction with poset games.
\newblock {\em Bull. Eur. Assoc. Theor. Comput. Sci. EATCS}, 116:42--75, 2015.

\bibitem[Gal74]{gale}
David Gale.
\newblock A curious nim-type game.
\newblock {\em Amer. Math. Monthly}, 81:876--879, 1974.

\bibitem[GMK18]{ALCO_2018__1_3_371_0}
Ignacio Garc{\'\i}a-Marco and Kolja Knauer.
\newblock Chomp on numerical semigroups.
\newblock {\em Algebraic Combinatorics}, 1(3):371--394, 2018.

\bibitem[Jec03]{Jech2003}
Thomas Jech.
\newblock {\em Set theory.}
\newblock Springer Monogr. Math. Berlin: Springer, the third millennium edition, revised and expanded edition, 2003.
\end{thebibliography}
\end{document}